\documentclass[11pt]{article}
\usepackage{latexsym, amssymb, amsfonts,amsmath}
\usepackage{amssymb,latexsym}
\usepackage{amsfonts}
\usepackage{amsmath,amsthm,graphicx}
\usepackage[a4paper,left=3cm,width=14cm,right=3cm]{geometry}

\newcommand{\be}{\begin{equation}}
\newcommand{\ee}{\end{equation}}
\newcommand{\bea}{\begin{eqnarray}}
\newcommand{\eea}{\end{eqnarray}}
\newcommand{\bean}{\begin{eqnarray*}}
\newcommand{\eean}{\end{eqnarray*}}
\newcommand{\brray}{\begin{array}}
\newcommand{\erray}{\end{array}}
\newcommand{\ben}{\begin{equation}{nonumber}}
\newcommand{\een}{\end{equation}{nonumber}}

\newtheorem{dfn}{Definition}[section]
\newtheorem{thm}[dfn]{Theorem}
\newtheorem{lmma}[dfn]{Lemma}

\newtheorem{ppsn}[dfn]{Proposition}
\newtheorem{crlre}[dfn]{Corollary}
\newtheorem{xmpl}[dfn]{Example}
\newtheorem{rmrk}[dfn]{Remark}
\newcommand{\bdfn}{\begin{dfn}}
\newcommand{\bthm}{\begin{thm}}

\newcommand{\blr}{\begin{list}{$($\roman{cnt1}$)$} {\usecounter{cnt1}
        \setlength{\topsep}{0pt} \setlength{\itemsep}{0pt}}}
\newcommand{\bla}{\begin{list}{$($\alph{cnt2}$)$} {\usecounter{cnt2}
       \setlength{\topsep}{0pt} \setlength{\itemsep}{0pt}}}
\newcommand{\bln}{\begin{list}{$($\arabic{cnt3}$)$} {\usecounter{cnt3}
                \setlength{\topsep}{0pt} \setlength{\itemsep}{0pt}}}
\newcommand{\el}{\end{list}}

\newcommand{\blmma}{\begin{lmma}}
\newcommand{\bppsn}{\begin{ppsn}}
\newcommand{\bcrlre}{\begin{crlre}}
\newcommand{\bxmpl}{\begin{xmpl}}
\newcommand{\brmrk}{\begin{rmrk}}
\newcommand{\edfn}{\end{dfn}}
\newcommand{\ethm}{\end{thm}}
\newcommand{\elmma}{\end{lmma}}

\newcommand{\eppsn}{\end{ppsn}}
\newcommand{\ecrlre}{\end{crlre}}
\newcommand{\exmpl}{\end{xmpl}}
\newcommand{\ermrk}{\end{rmrk}}

\newcommand{\IB}{{I\! \! B}}
\newcommand{\IC}{\mathbb{C}}

\newcommand{\IN}{{I\! \! N}}

\newcommand{\IR}{\mathbb{R}}

\newcommand{\innerl}{\left\langle}
\newcommand{\innerr}{\right\rangle}

\newcommand{\cla}{{\cal A}}
\newcommand{\clb}{{\cal B}}

\newcommand{\cld}{{\cal D}}
\newcommand{\cle}{{\cal E}}

\newcommand{\clh}{{\cal H}}

\newcommand{\clk}{{\cal K}}
\newcommand{\cll}{{\cal L}}

\newcommand{\clp}{{\cal P}}

\newcommand{\clt}{{\cal T}}
\newcommand{\clu}{{\cal U}}
\newcommand{\clv}{{\cal V}}
\newcommand{\clw}{{\cal W}}

\newcommand{\clz}{{\cal Z}}

\def\wA{\widetilde{A}}

\def\wC{\widetilde{C}}

\def\a*{{\cal A}_{h,*}}
\def\B{{\cal B}(h)}
\def\B1{{\cal B}_1(h)}
\def\b{{\cal B}^{\rm s.a.}(h)}
\def\b1{{\cal B}^{\rm s.a.}_1(h)}

\newcommand{\ot}{\otimes}

\begin{document}

 \begin{center}
 \Large{\bf{Dilation of arbitrary symmetric quantum dynamical semigroups on $B(\clh)$}}\\
{\large Biswarup Das{\footnote{biswarupnow@gmail.com}\footnote {Indian Statistical Institute, Kolkata.}} }\\
\end{center}
\begin{abstract}
We prove the existence of Hudson Parthasarathy dilation of a quantum dynamical semigroup on $B(\clh),$ which is symmetric with respect to the canonical normal trace on it.
\end{abstract}
\begin{center}
\end{center}

\section{Introduction}
Dilation of quantum dynamical semigroups (QDS) using quantum stochastic calculus is one of the most interesting and important problem of Quantum Probability (see \cite{accardi-lewis,dgkbs,krp,hudson-krp}). It is known that a QDS with bounded generator always admits Hudson-Parthasarathy (HP) dilation (see \cite{dgkbs,lindsay-wills}). Construction of such dilation amounts to solving quantum stochastic differential equation (QSDE) with bounded coefficients, and prescribed initial values and proving the unitarity of the solution. Such unitary solution always exists as long as the coefficients are bounded \cite{dgkbs,lindsay-wills}. For a QDS with unbounded generator, no such results are known in general. However, certain sufficient conditions on the unbounded operator coefficients for e.g. \cite[p.174]{dgkbs},\cite{fagnola,chebo-fagnola,sinha-mohari,sinha-mohari1,accardi-kozyrev1} are known using which one can solve QSDE with unbounded coefficients. Using these techniques, the authors of \cite{dgkbs} proved the existence of Hudson-Parthasarathy dilation  of symmetric QDS which are covariant with respect to the action of a Lie group \cite[Theorem 8.1.23]{dgkbs}. The key fact that allowed them to construct such dilations is the existence of a ``nice'' dense subspace within the domain of the adjoints of the coefficients. Such subspaces may not exist in general. In this paper, we will show that in context of $B(\clh),$ symmetry with respect to the canonical trace is sufficient to ensure the existence HP dilation of a QDS and hence the assumption of covariance is not required.  

\section{Notations and terminologies}
\subsection{Quantum Stochastic flows of Hudson-Parthasarathy type}
We shall refer the reader to \cite{krp,accardi-lewis,dgkbs,hudson-krp} and references therein for the basics of quantum stochastic calculus. We will adopt the Hudson-Parthasarathy formalism of quantum stochastic calculus, which we very briefly review here. We will consider the coordinate free version of the quantum stochastic calculus, as discussed in \cite{dgkbs}. All the Hilbert spaces appearing in this article will be separable and for 
a Hilbert space $\clh$ we shall denote by $\Gamma(\clh)$ and $\Gamma^f(\clh),$ the symmetric and free Fock space over $\clh$ respectively. ${\rm Lin}(\clv,\clw)$ will denote the space of linear(possibly defined on a subset of $\clv$) maps from a vector space $\clv$ to another vector space $\clw$. By $Dom(L)$ and $Ran(L),$ we will denote respectively the domain and range of a possibly unbounded operator $L$ on a Banach space.

\bdfn
We say that a family of maps $(X_t)_{t\geq0}$ belonging to\\ $Lin(h\ot\Gamma,h\ot\Gamma)$ (where $\Gamma:=\Gamma(L^2(\IR_+,k_0))$), is a Hudson-Parthasarathy flow, where $k_0$ is the Hilbert space ($1\leq dim~k_0\leq\infty$) of noise or multiplicity, if it satisfies a quantum stochastic differential equation (QSDE for short) of the form :
\be\label{qsde}
dX_t=X_t(a_R(dt)+a^\dagger_S(dt)+\Lambda_T(dt)+Adt)
\ee
with prescribed initial value $X_0=\tilde{X}_0\ot1$ where $\tilde{X}_0\in B(h),$ $R,S\in Lin(h,h\ot k_0),$\\ $T\in Lin(h\ot k_0,h\ot k_0)$ and $A\in Lin(h,h).$\\
\edfn
Here $a_R(dt),~a^\dagger_S(dt),~and~\Lambda(dt),$ are the annihilation, creation and number fields respectively.
We refer the reader to \cite{dgkbs} for a detailed discussion on the coordinate free formalism of quantum stochastic calculus.

We note that the above QSDE has to be interpreted as the strong integral equation: 
$$X_t(ve(g))=X_0(ve(g))+\left(\int_0^t X_s(a_R(ds)+a^\dagger_S(ds)+\Lambda_T(ds)+Ads\right)(ve(g)),$$ for all $v\in h,~g\in L^2(\IR_+,k_0)$ with $\int_0^t\|g(s)\|^4ds<\infty$ for all $t.$ $~R,S,T~and~A$ will be called the coefficients associated with a QSDE of the form (\ref{qsde}).\\

\subsection{Quantum Dynamical Semigroup}
Let $\cla$ be a unital $C^*$ or von-Neumann algebra. A semigroup of bounded operators $(T_t)_{t\geq0}$ on $\cla$ will be called a quantum dynamical semigroup (QDS for short), if $T_t$ is a strongly continuous (in the norm or ultraweak topology accordingly as $\cla$ is a $C^*$ or von-Neumann algebra), completely positive real map for each $t\geq0.$ A QDS is called conservative if $T_t(1)=1.$

Let $\cll$ be the generator of a QDS $(T_t)_{t\geq0}$ on $\cla\subseteq B(\clh).$ $\cll$ is said to have a Christensen-Evans form if there exists a Hilbert space $\clk,$ a densely defined operator $R\in Lin(\clh,\clk),$ a representation $\rho:\cla\rightarrow B(\clk)$ and a self-adjoint operator $H\in Lin(\clh,\clh)$ such that

\begin{equation*}
\cll(x)=R^*\rho(x)R-\frac{1}{2}(R^*R-\cll(1))x-\frac{1}{2}x(R^*R-\cll(1))+i[H,x]~;
\end{equation*}
for all $x\in Dom(\cll).$ It is known (see \cite{chris,lewis}) that if $\cll$ is a bounded operator, it has Christensen-Evans form.

Conversely given a map which has Christensen-Evans form, under some additional hypotheses one can construct a minimal QDS on $\cla,$ associated to this map. We refer to \cite[p.39]{dgkbs},~\cite{sinha-mohari1,chebo-fagnola} for more discussions on minimal semigroup.
\bdfn
Let $\cla$ be a $C^*$ or von-Neumann algebra equipped with a faithful, lower-semi-continuous trace $\tau.$ Suppose that $(T_t)_{t\geq0}$ is a QDS on $\cla.$ Then $(T_t)_{t\geq0}$ is said to be symmetric with respect to $\tau$ if
\begin{equation*}
\tau(T_t(x)y)=\tau(xT_t(y))~\mbox{for all $x,y$ such that $T_t(x)y\in Dom(\tau).$}
\end{equation*}
\edfn
A symmetric QDS always extends to a $C_0$ semigroup of self-adjoint operators in the Hilbert space $L^2(\tau).$ We will denote by $\cll_2,$ the generator of the $L^2(\tau)$-extension of the symmetric QDS. We refer the reader to \cite[p.61-68]{dgkbs},~\cite{cipriani} for a detailed discussion of symmetric QDS. Theorem 3.2.30 in page 65 of \cite{dgkbs} and Theorem 3.2.31 in page 68 of \cite{dgkbs} together implies that the generator of a symmetric QDS, under some conditions, has a Christensen-Evans form.
\subsection{Hudson-Parthasarathy dilation of a quantum dynamical semigroup}
\bdfn
A Hudson-Parthasarathy dilation (HP dilation for short) of a QDS $(T_t)_{t\geq0}$ on a $C^*$ or von-Neumann algebra $\cla\subseteq B(h)$ is given by a family $(U_t)_{t\geq0}$ of unitary operators acting on $h\ot\Gamma,$ such that the following holds:
\begin{enumerate}
\item[(i)]
$U_t$ satisfies a QSDE of the form (\ref{qsde}) with initial condition $U_0=I.$
\item[(ii)]
For all $u,v\in h,~x\in\cla,$ $$\innerl ve(0),U_t(x\ot I)U_t^*ue(0)\innerr=\innerl v,T_t(x)u\innerr.$$
\end{enumerate}
\edfn
It is known that QDS with bounded generator always admits HP dilation (see \cite{dgkbs,lindsay-wills}). Some partial results are also known for QDS with unbounded generator (see \cite{dgkbs,chebo-fagnola,accardi-kozyrev1}).

The main goal of this paper is to prove the following theorem:

\bthm\label{main theorem}
Suppose $(T_t)_{t\geq0}$ is a conservative, symmetric QDS on $B(h)$ (symmetric with respect to the canonical trace), with ultraweak generator $\cll.$ Then $(T_t)_{t\geq0}$ always admits an HP dilation.
\ethm

Before proving Theorem \ref{main theorem}, we recall some facts about unbounded derivations in the next section. We refer the reader to \cite{bratteli} for more discussions on the topic.

\section{Unbounded derivations.}
For a Hilbert space $\clh,$ let $K(\clh)$ denote the space of compact operators on $\clh.$ A derivation $\delta\in Lin(\cla,\cla),$ where $\cla$ is a $\ast$-algebra, is called symmetric if $\delta(A^*)=\delta(A)^*.$

\begin{ppsn}\cite[p.238]{bratteli}\label{derivation as commutator}
Let $\delta$ be a symmetric derivation defined on a  $\ast$-subalgebra $\cld$ of the bounded operators in a Hilbert space $\clh.$ Let $\Omega\in\clh$ be a unit vector, cyclic for $\cld$ in $\clh$ and denote the corresponding state by $\omega$ (i.e. $\omega(x)=\innerl\Omega,x\Omega\innerr$). Suppose we have $|\omega(\delta(A))|\leq L\{\omega(A^*A)+\omega(AA^*)\}^{\frac{1}{2}}$ for some constant $L.$ Then there exists a symmetric operator $H$ on $\clh$ such that
\begin{equation*}
\begin{split}
&Dom(H)=Dom(\delta)\Omega,\\
&\delta(A)\psi=i[H,A]\psi~for~\psi\in\clh;
\end{split}
\end{equation*}
where $[H,x]:=Hx-xH.$
\end{ppsn}

\begin{lmma}\label{matrix convergence}
Let $({\bf A})_p,\IB\in M_n(\IC),$ such that ${\bf A}_p\rightarrow\IB$ as $p\rightarrow\infty.$ Suppose that $\lambda\in\IC$ is an eigenvalue of $\IB$ with multiplicity $m.$ Then given a small neighbourhood $U$ of $\lambda,$ there exists $p_0(U)\in\IN$ such that for all $p\geq p_0(U),$ ${\bf A}_p$ will have $m$ eigenvalues in the neighbourhood $U.$
\end{lmma}
\begin{proof}
Observe that $|det({\bf A}-z{\bf I}_n)-det(\IB-z{\bf I}_n)|\rightarrow0.$ The result now follows by an application of Hurwitz's theorem.
\end{proof}

\begin{lmma}\label{lmma}
Let $\delta$ be a symmetric derivation on $\clb(\clh),$ such that $Dom(\delta)$ is dense in the weak operator topology. Assume that $Dom(\delta)$ is closed under holomorphic functional calculus and $Dom(\delta)\cap K(\clh)\neq\{0\}.$ Then $Dom(\delta)$ contains a rank-one projection.
\end{lmma}
\begin{proof}
The proof is an adaptation of the arguments given in \cite{bratteli}:

Suppose $B\in Dom(\delta)\cap K(\clh)$ and $C=B^*B.$ Choose an eigenvalue $\lambda$ of $C$ and let $E_\lambda$ be the associated finite rank spectral projection. Then
\begin{equation*}
E_\lambda=\frac{1}{2\pi i\lambda}\int_\Gamma d\gamma~ C(\gamma-C)^{-1},
\end{equation*}
where $\Gamma$ is such that it contains the isolated point $\lambda.$ Thus $E_\lambda=\frac{1}{\lambda}Cf(C),$ where $f(\cdot)$ is the holomorphic function $f(z)=z,~z\in\IC.$ Thus $E_\lambda\in Dom(\delta).$ Now choose a projection $P$ such that $E_\lambda PE_\lambda=P.$ Get $A_n\in Dom(\delta)$ such that $A_n=A_n^*$ and $A_n\stackrel{SOT}{\rightarrow}P.$  We have
$E_\lambda A_n E_\lambda\stackrel{SOT}{\rightarrow}E_\lambda PE_\lambda$ which implies that
$\|E_\lambda A_nE_\lambda-E_\lambda PE_\lambda\|\rightarrow0$ since the $C^*$-algebra $E_\lambda \clb(\clh)E_\lambda$ is of finite dimension. Thus for large $n,$ $E_\lambda A_nE_\lambda$ has a simple eigenvalue in a neighbourhood around $1,$ by lemma \ref{matrix convergence}. Let $E$ be the finite-rank projection. Then considering a curve around that simple eigenvalue, we can conclude that $E\in Dom(\delta)$ by a similar argument.
\end{proof}

\begin{lmma}\label{observation after Theorem W}
Let $\delta$ be a symmetric derivation satisfying the hypotheses of Lemma \ref{lmma}. Then there exists a symmetric operator $H$ on $\clh$ such that $Dom(H):=Dom(\delta)\Omega$ and $\delta(x)=i[H,x]$ for all $x\in Dom(\delta),$ for some $\Omega\in\clh,~\|\Omega\|=1.$
\end{lmma}
\begin{proof}
Let $E$ be the finite rank projection as~ obtained in Lemma \ref{lmma}. Suppose that $\Omega\in Ran(E)$ such that $\|\Omega\|=1.$ Let $\omega(x)=\innerl \Omega, x\Omega\innerr.$ Then

\begin{equation*}
\begin{split}
|\omega(\delta(A))|&=|\omega(E\delta(A)E)|\\
&\leq|\omega(\delta(EAE))|+|\omega(\delta(E)A)|+|\omega(A\delta(E))|\\
&\leq 3\|\delta(E)\|~[\omega(A^*A)+\omega(AA^*)]^{\frac{1}{2}};
\end{split}
\end{equation*}
so that by Proposition \ref{derivation as commutator}, we have the required result.
\end{proof}
\begin{ppsn}\cite{Naimark1, Naimark2}\label{Naimark}
If $H$ is a densely defined symmetric operator on $\clh$ such that $dim(H-iI)^{\perp}\neq dim(H+iI)^{\perp}.$ Then there exists a Hilbert space $\widehat{\clh}\supseteq\clh$ and a self-adjoint operator $K$ acting on $\widehat{\clh}$ such that $K|_{_{\clh}}=H$ and we have the integral representation
\begin{equation*}
\innerl Hu,v \innerr=\int_{-\infty}^{\infty}t~d\innerl F_tu,v\innerr;
\end{equation*}
for $u\in Dom(H),$ $v\in\clh;$ where $F_t$ is the generalized resolution of identity.
\end{ppsn}

\section{Existence of HP-dilation}\label{HP dilation}
Let $tr_{_{\clh}}$ denote the canonical trace of $\clb(\clh).$ Observe that in this case, the Hilbert space  $L^2(tr_{_{\clh}})$ is identified with the space of Hilbert-smith operators on $\clh$ which we denote by $B_2(\clh).$ Let $\cll_2$ denote the restriction of $\cll$ to $B_2(\clh)$ and $\clb$ denote the associated Dirichlet algebra. Clearly $\clb=Dom((-\cll_2)^{\frac{1}{2}}).$ Note that with respect to the $C^*$-subalgebra $K(\clh),$ the Dirichlet form associated with the semigroup is a $C^*$-Dirichlet form since the $\ast$-subalgebra $\clb$ is norm dense in $K(\clh).$ So the set of results in \cite[p.84-p.89,~p.91,~p.96.]{cipriani} gives the following:
\begin{itemize}
\item There exists a $K(\clh)-K(\clh)$ Hilbert-bi-module $\clk$ and a $\pi$-derivation $\delta_0:\clb\rightarrow\clk$ such that 
\begin{equation*}
\begin{split}
&\innerl\delta_0(x),\delta_0(y)e\innerr_{\clk}=\lim_{\epsilon\rightarrow0}tr_{_{\clh}}\left(K_{\epsilon}(x,y)e\right),\\&\mbox{where}~K_\epsilon(x,y)=\cll_2(1-\epsilon\cll_2)^{-1}(x^*y)-\cll_2(1-\epsilon\cll_2)^{-1}(x^*)y-x^*\cll_2(1-\epsilon\cll_2)^{-1}(y),
\end{split}
\end{equation*}
for $x,y\in\clb,e\in K(\clh)$ where $\pi$ is the left action of $K(\clh)$ on $\clk.$
\item $\delta_0$ viewed as an element of $Lin(L^2(tr_{_{\clh}}),\clk))$ denoted by $R_0,$ has domain $\clb$ such that
\begin{equation}\label{R_0}
\innerl R_0(x),R_0(y)\innerr_\clk=-tr_{_{\clh}}(2\cll(x^*)y),~\mbox{for $x\in Dom(\cll_2)\subseteq\clb$,}
\end{equation}
so that vectors of $\clk$ of the form $\delta_0(x)$ for $x\in Dom(\cll_2)$ belongs to $Dom(R_0^*)$ and hence
\begin{equation}\label{cll_2}
\cll_2=-\frac{1}{2}R_0^*R_0|_{_{Dom(\cll_2)}}.
\end{equation}
We also have $\delta_0(x)y=(R_0x-\pi(x)R_0)y$ for $x,y\in\clb$.
\end{itemize}
Without loss of generality, we may suppose that $\pi:K(\clh)\rightarrow \clb(\clk)$ is a non-degenerate $C^*$-representation. Thus it is equal to a direct sum of irreducibles which are unitarily equivalent to the identity representation.  So $\pi$ extends to $\clb(\clh)$ as a unital normal $\ast$-representation, which we again denote by $\pi.$ Now by GNS construction with respect to $tr_{_{\clh}}$, $\clb(\clh)\subseteq \clb(L^2(tr_{_{\clh}})).$ Thus there exists an isometry $\Sigma:\clk\rightarrow L^2(tr_{_{\clh}})\ot k_0,$ for some separable Hilbert space $k_0$ such that $\pi(x)=\Sigma^*(x\ot1_{k_0})\Sigma$ and $\Sigma\Sigma^*$ commutes with $(x\ot1_{k_0})$ (\cite[Chapter 2]{dgkbs}). Then $\delta:=\Sigma\delta_0$ satisfies $\delta(xy)=\delta(x)y+(x\ot1_{k_0})\delta(y).$ Moreover, equations (\ref{R_0}) and (\ref{cll_2}) hold with $R_0$ replaced by $R:=\Sigma R_0$ and we have the identity
$\delta(x)y=(Rx-(x\ot 1_{k_0})R)y$ for $x,y\in\clb.$ Note that here $\delta:\clb\rightarrow \clb(\clh)\ot k_0,$ since we have 
$L^2(tr_{_{\clh}})\equiv B_2(\clh)\subseteq \clb(\clh)$ and in this case $\|\cdot\|_\infty\leq\|\cdot\|_2.$

Let $\clv_0:=\{\sum_{i=1}^k\lambda_i e_i:~(e_i)_{i}~\mbox{is an orthonormal basis for $k_0$ and $\lambda_i\in\IC$}\}.$

We make an easy observation at this point:

\begin{lmma}\label{LCT}
Suppose that $(\clt_t)_{t\geq0}$ is an ultraweakly continuous $C_0$ (in the ultraweak topology) contractive semigroup on $\clb(\clh),$ with generator $C.$ Let $\clz\subseteq\clb(\clh)$ is a subspace of $\clb(\clh)$ which is closed with respect to a locally convex topology (LCT for short) given by a family of seminorms, say $(p_\alpha)_{\alpha}.$ Suppose that $(\clt_t)_{t\geq0}$ restricted to $\clz$ becomes a $C_0$ (with respect to the LCT described above) semigroup, with generator say $\wC.$ Then if $x\in Dom(C)\cap\clz$ such that $C(x)\in\clz,$ then $x\in Dom(\wC)$ and $\wC(x)=C(x).$
\end{lmma}

\begin{proof}
Let $x\in Dom(C)\cap \clz.$ Since $(\clt_t)_{t\geq0}$ is a $C_0$ semigroup with respect to the ultraweak topology of $\clb(\clh),$ we have $\clt_t(x)-x=\int_0^t ds~ \clt_s(C(x)),$ where the integral in convergent in the ultraweak topology. Moreover, as $\clt_t$ is a contraction for each $t\geq0,$ $\int_0^t ds~\clt_s(C(x))\in\clb(\clh).$ But we have $C(x)\in \clz$ and as $(\clt_t)_{t\geq0}$ restricted to $\clz$ is a $C_0$ semigroup with respect to the LCT described in the hypothesis, the integral also converges in this LCT, which implies that $\frac{\clt_t(x)-x}{t}$ converges in this LCT to $C(x).$ Thus we have the required result.  
\end{proof}

\begin{lmma}\label{preparation of Mohari Sinha}
The $\ast$-subalgebra $\clb\subseteq Dom(\innerl\xi,R\innerr^*)$ for $\xi\in\clv_0.$
\end{lmma}

\begin{proof}
Note that $\delta:\clb\rightarrow \clb(\clh)\ot k_0$ is a derivation satisfying the identity\\
$\delta(xy)=\delta(x)y+(x\ot 1_{k_0})\delta(y)$ for $x,y\in\clb.$ Let us define $\theta^i_0(\cdot):=\innerl e_i,\delta(\cdot)\innerr.$ Since $\delta$ is a derivation, it follows that $\theta^i_0(\cdot)$ is a derivation and $Dom(\theta^i_0)=\clb,$ for each $i\in\IN.$ We prove that $tr_{_{\clh}}(\theta^i_0(xy))=0$ for all $x,y\in\clb,i\in\IN,$ which will imply the result, because then $(\theta^i_0)^\ast\supseteq\theta^0_i(:=(\theta^i_0)^\dagger,$ defined below.)

Fix an $i\in\IN.$
Recall that in our case, $\clb=Dom((-\cll_2)^\frac{1}{2})\subseteq B_2(\clh).$ Let us define two new derivations $\delta_1:=\frac{\theta^i_0+\theta^{i\dagger}_0}{2}$ and $\delta_2:=\frac{\theta^i_0-\theta^{i\dagger}_0}{2i},$ where $\theta^{i\dagger}_0(x)=(\theta^i_0(x^*))^*.$ Then we have $\delta=\delta_1+i\delta_2$ and $Dom(\delta_1)=Dom(\delta_2)=\clb.$ Moreover, $\delta_1$ and $\delta_2$ are symmetric derivations. By the results in page.103 of \cite{cipriani}, $\clb$ is closed under $C^1$ functional calculus and thus it is closed under holomorphic functional calculus. So by Lemma \ref{lmma}, $Dom(\delta_1)$ contains a finite rank operator. Hence by Lemma \ref{observation after Theorem W}, $\delta_1(x)=i[T,x]$ for some symmetric operator $T$ acting on $\clh$ and we have $Dom(T):=Dom(\delta_1)\Omega,$ where $\Omega\in\clh$ is cyclic for $Dom(\delta_1).$ Now suppose $dim(T-iI)^\perp\neq dim(T+iI)^\perp.$ Let $K$ denotes the self-adjoint extension of $T$ as described in Proposition \ref{Naimark}, so that $K=K^*.$ Let $P:\widehat{\clh}\rightarrow\clh$ be the orthogonal projection. Let $\widehat{\clh}$ be decomposed in the basis of $P$ i.e. $\widehat{\clh}=\clh\oplus\clh^\perp.$ With respect to this decomposition, an operator $S\in \clb(\widehat{\clh})$ can be viewed as a matrix
$\begin{pmatrix}
S_{11}&S_{12}\\
S_{21}&S_{22}
\end{pmatrix},$ where $S_{11}\in \clb(\clh),~S_{12}\in \clb(\clh^\perp,\clh),~S_{21}\in \clb(\clh,\clh^\perp)$ and $S_{22}\in \clb(\clh^\perp).$ Moreover, if $tr_{_{\widehat{\clh}}},~tr_{_{\clh}}$ and $tr_{_{\clh^\perp}}$ denote the canonical traces of the operator algebras $\clb(\widehat{\clh}),~\clb(\clh)$ and $\clb(\clh^\perp)$ respectively, then we have $tr_{_{\widehat{\clh}}}(S)=tr_{_{\clh}}(S_{11})+tr_{_{\clh^\perp}}(S_{22}).$ Consider the ultraweakly continuous  $C_0$ automorphism group $(\alpha_t)_{t\in\IR}$ defined by $\alpha_t(X)=e^{itK}Xe^{-itK}$ for $X\in \clb(\widehat{\clh}).$ Let $A$ denote the generator of the semigroup $(\alpha_t)_{t\geq0}$. Then we have $A(X)=i[K,X],$ for $X\in Dom(A).$ 
Now note that $tr_{_{\widehat{\clh}}}(\alpha_t(x))=tr_{_{\widehat{\clh}}}(x)$ for $x\geq0.$ Thus $(\alpha_t)_{t\in\IR}$ restricted to $L^2(tr_{_{\widehat{\clh}}})$ becomes a contractive group of unitary operators on $L^2(tr_{_{\widehat{\clh}}}),$ which we denote by $(\clu_t)_{t\in\IR}.$ Let $\clp:=\innerl|u\innerr\innerl v|:~u,v\in Dom(K)\innerr_{_{\IC}}.$ Then it follows that $\lim_{t\rightarrow0}\clu_t(X)\stackrel{L^2(tr_{_{\widehat{\clh}}})}{\longrightarrow}X,$ for all $X\in\clp.$ Moreover, as $\clp$ is dense in $L^2(tr_{_{\widehat{\clh}}})$ and $\clu_t$ is a contraction operator on $L^2(tr_{_{\widehat{\clh}}})$ for each $t\in\IR,$ it follows that $(\clu_t)_{t\geq0}$ is a $C_0$ semigroup of operators in $L^2(tr_{_{\widehat{\clh}}}).$ Let its generator be denoted by $\wA.$ Note that $\wA$ is also a derivation. It is easy to see that $L^2(tr_{_{\widehat{\clh}}})=L^2(tr_{_{\clh}})\oplus L^2(tr_{_{\clh^\perp}}).$ Now $\clb\in L^2(tr_{_{\clh}})$ and $A(x)=\delta_1(x)\in L^2(tr_{_{\clh}})$ for $x\in\clb$ and thus by Lemma \ref{LCT}, we have $\clb\subseteq Dom(\wA)$ and $A(x)=\wA(x)=\delta_1(x).$  Furthermore, we have $tr_{_{\widehat{\clh}}}(A(XY))=0$ for $X,Y\in Dom(\widetilde{A}).$  So we have
\begin{equation*}
tr_{_{\widehat{\clh}}}(A(XY))=tr_{_{\clh}}(A(XY))=tr_{_{\clh}}(\delta_1(XY))=0,
\end{equation*}
for all $X,Y\in \clb.$ Likewise, one may prove $tr_{_{\clh}}(\delta_2(XY))=0$ for $X,Y\in \clb.$ Thus we have $tr_{_{\clh}}(\theta^i_0(xy))=0$ for $x,y\in \clb.$ Observe that if the deficiency indices of $T,$ i.e. the numbers $dim(T-iI)^\perp$ and $dim(T+iI)^\perp$ are equal, then $T$ has a self-adjoint extension which belongs to $Lin(\clh,\clh).$ Then we may repeat the same argument as above and reach the same conclusion. Hence the lemma is proved.
\end{proof}

\begin{lmma}\label{final touch}
$Dom(\cll_2)$ is a $\ast$-subalgebra.
\end{lmma}
\begin{proof} 
The QDS $(T_t)_{t\geq0}$ is $\ast$-preserving i.e. $T_t(x^*)=(T_t(x))^*$ for each $t\geq0$ and $x$ belonging to $\clb(\clh).$ Thus $Dom(\cll_2)$ is a $\ast$-closed subspace. We prove that $\theta^{i\ast}_0\in Lin(L^2(tr_{_{\clh}}),L^2(tr_{_{\clh}}))$ is a derivation and $\theta^{i\ast}_0(x)=-\theta^0_i(x),~\forall~x\in\clb,$ where $\theta^0_i:=(\theta^i_0)^\dagger$ as follows:

It also follows immediately by using $tr_{_{\clh}}(\theta^i_0(xy))=0$ for $x,y\in\clb$ that $\theta^{i\ast}_0(\cdot)|_{_{\clb}}=\theta^0_i(\cdot).$ Let $x\in Dom(\theta^{\ast i}_0).$ Now note that $\theta^i_0(x)=\delta_1(x)+i~\delta_2(x),~x\in\clb,$ where $\delta_l(\cdot)$ is a symmetric derivation for each $l=1,2.$ Let $\delta(x)=i[T_1,x]$ and $\delta_2(x)=i[T_2,x],$ where $T_1,T_2$ are the symmetric operators obtained by lemma \ref{observation after Theorem W}. Since $\clb\subseteq Dom(\delta_1)\cap Dom(\delta_2),$ following the proof of lemma \ref{lmma}, we see that we can select a common vector $\Omega\in\clh$ such that $\clb\Omega=Dom(T_1)=Dom(T_2).$ Now note that $\theta^i_0(x)=[T,x],$ where $T:=i T_1-T_2.$ It is enough to prove that $\theta^{\ast i}_0(x)=[S,x],$ where $S=-iT_1-T_2.$ We prove this as follows:

Let $\cld:=\innerl|u_1\innerr\innerl u_2|:~u_1,u_2\in\clb\Omega\innerr_{\IC}\subseteq Dom(\delta_1)\cap Dom(\delta_2)\cap L^2(tr_{_{\clh}}).$ Moreover $\clb\Omega$ is dense in $\clh.$  Since 
$\innerl x,\theta^i_0(y)\innerr=\innerl \theta^{\ast i}_0(x),y\innerr$ for all $y\in Dom(\theta^i_0)\cap L^2(tr_{\clh}),$ in particular for $y\in\cld,$ it follows that 
$[S,x]\in \clb(\clh)$ for $x\in Dom(\theta^{\ast i}_0),$ $S$ as described above and hence proved.

Now we have $Dom(\cll_2)\subseteq\cap_{i\geq1}Dom(\theta^{i\ast}_0\theta^i_0)$ and $\sum_{i\geq1}\|(\theta^{i\ast}_0\theta^i_0)x\|_2<\infty$ for $x\in Dom(\cll_2).$ The fact that 
$\theta^i_0(Dom(\cll_2))\subseteq Dom(\theta^{i\ast}_0)$ implies that if $x,y\in Dom(\cll_2),$ then $xy$ belongs to $Dom(\theta^{i\ast}_0\theta^i_0)$ i.e. $xy\in Dom(\theta^{i\ast}_0\theta^i_0)$ for each $i.$ To prove that\\ $xy\in Dom(\cll_2),$ we just need to show that $\sum_{i\geq1}\|(\theta^{i\ast}_0\theta^i_0)xy\|_2<\infty.$ Now for each $i,$ we have 
\begin{equation}\label{turuper-tash}
\begin{split}
(\theta^{i\ast}_0\theta^i_0)xy&=\theta^{i\ast}_0(\theta^i_0(x)y+x\theta^i_0(y)),\\
&=\theta^{i\ast}_0\theta^i_0(x)y+x\theta^{i\ast}_0\theta^i_0(y)+
\theta^i_0(x)\theta^{i\ast}_0(y)+\theta^{i\ast}_0(x)\theta^i_0(y);\\
&=\theta^{i\ast}_0\theta^i_0(x)y+x\theta^{i\ast}_0\theta^i_0(y)-
\theta^i_0(x)\theta^0_i(y)-\theta^0_i(x)\theta^i_0(y),~\mbox{since $x,y\in Dom(\cll_2)\subseteq\clb.$}
\end{split}
\end{equation}
Observe that $\sum_{i\geq1}\|(\theta^{i\ast}_0\theta^i_0)(x)\|_2<\infty$ and $\sum_{i\geq1}\|(\theta^{i\ast}_0\theta^i_0)(y)\|_2<\infty$ since $x,y\in Dom(\cll_2).$ Now
\begin{equation*}
\begin{split}
\|\theta^i_0(x)\theta^0_i(y)\|_2&=
\sqrt{tr_{_{\clh}}((\theta^0_i(y))^*(\theta^i_0(x))^*\theta^i_0(x)\theta^0_i(y))}\\
&\leq\|\theta^0_i(y)\|_2~\sqrt{tr((\theta^i_0(x))^*\theta^i_0(x))}~\mbox{since $\|\cdot\|_\infty\leq\|\cdot\|_2$};
\end{split}
\end{equation*}
so by an application of the Cauchy-Schwartz inequality, we have $\sum_{i\geq1}\|\theta^i_0(x)\theta^0_i(y)\|_2<\infty$ Similarly it can be proved that $\sum_{i\geq1}\|\theta^i_0(y)\theta^0_i(x)\|_2<\infty.$ So we have\\ 
$\sum_{i\geq1}\|(\theta^0_i\theta^i_0)xy\|_2^2<\infty$ which proves the lemma.
\end{proof}
Note that $Dom(\cll_2)$ becomes a $\ast$-subalgebra which is ultraweakly dense in $\clb(\clh)$ as well as dense in $L^2(tr_{_{\clh}})$ (i.e. in the norm $\|\cdot\|_2$). $Dom(\cll_2)$ is also a core for the Dirichlet form $\cle(\cdot,\cdot).$ Furthermore we have $T_t(Dom(\cll_2))\subseteq Dom(\cll_2).$ Thus it is also a core for $\cll.$ Moreover, since $Dom(\cll_2)$ is an algebra, we have 
\begin{equation*}
\begin{split}
&\delta(x)^*\delta(y)=\cll(x^*y)+\cll(x^*)y-x^*\cll(y),\\ 
&\cll(x)=R^*(x\ot 1_{k_0})R-\frac{1}{2}R^*Rx-\frac{1}{2}xR^*R,
\end{split}
\end{equation*}
for $x,y\in Dom(\cll_2)$ (see \cite[Chapter 3]{dgkbs}).
We now return to the proof of the main Theorem.
\bthm
 Suppose $(T_t)_{t\geq0}$ is a conservative QDS on $B(\clh)$ which is symmetric with respect to the canonical trace on $B(\clh).$ Let $\cll$ be the ultraweak generator of $(T_t)_{t\geq0}$ and $\cll_2$ be the generator of the $L^2$ extension of $(T_t)_{t\geq0}.$ Then $(T_t)_{t\geq0}$ always admits HP dilation.
\ethm

\begin{proof}
Consider the following QSDE:
\begin{equation}\label{QSDE}
\frac{dV_t}{dt}=V_t\circ(a^\dagger_\delta(dt)-a_\delta(dt)-\frac{1}{2}R^*R dt),
\end{equation}
with the initial condition $V_0=id.$ We will prove that there exists an unitary co-cycle $(U_t)_{t\geq0}$ which is a solution for the above QSDE. The coefficient matrix associated with the above QSDE is
$\clz=\begin{pmatrix}
-\frac{1}{2}R^*R&-R^*\\
R&0
\end{pmatrix}.$

Let $G_n=(1-\frac{\cll_2}{n})^{-1},$
$\clz^{(n)}=\begin{pmatrix}
-\frac{1}{2}G_nR^*R G_n&-G_nR^*\\
R G_n&0
\end{pmatrix}$ and $(e_i)_{i\in\IN}$ be an orthonormal basis for $k_0.$ For $\xi\in \clv_0,$ suppose $\hat{\xi}:=1\oplus\xi.$ We first prove that for $\omega\in Dom(\cll_2),$ $sup_{n\geq1}\|\clz^{(n)}_{\hat{\xi}}\omega\|^2<\infty.$ We have

\begin{equation*}
\begin{split}
\|R G_n\omega\|&=\innerl R G_n\omega,R G_n \innerr\\
&=\innerl\omega,G_n^*(-2\cll_2)G_n\omega\innerr\\
&=\innerl\omega,(-2\cll_2)^{\frac{1}{2}}G_n^*G_n(-2\cll_2)^{\frac{1}{2}}\omega\innerr\\
&=\|G_n(-2\cll_2)^{\frac{1}{2}}\omega\|^2\\
&\leq\|(-2\cll_2)^{\frac{1}{2}}\omega\|^2.
\end{split}
\end{equation*}
By Lemma \ref{preparation of Mohari Sinha} we have $\omega\xi:=\omega\ot\xi\in Dom(R^*).$ Thus

\begin{equation*}
\begin{split}
\|\clz^{(n)}_{\hat{\xi}}\omega\|^2&=\|-\frac{1}{2}G_nR^*R G_n\omega+G_nR^*(\omega\xi)\|^2
+\|R G_n\omega\|^2\\
&\leq 2\|G_n^2(-2\cll_2)\omega\|^2+2\|G_nR^*(\omega\xi)\|^2+\|R G_n\omega\|^2\\
&\leq 2\|(-2\cll_2)\omega\|^2+\|(-2\cll_2)^{\frac{1}{2}}\omega\|^2+2\|R^*(\omega\xi)\|^2;
\end{split}
\end{equation*}
which implies that $sup_{n\geq1}\|\clz^{(n)}_{\hat{\xi}}\omega\|<\infty.$ We next prove the following:

\begin{equation}\label{first limit}
\lim_{n\rightarrow\infty}\innerl\hat{\eta},\clz^{(n)}_{\hat{\xi}}\omega\innerr=\innerl\hat{\eta},\clz_{\hat{\xi}}\omega\innerr,
\end{equation}
for $\omega\in Dom(\cll_2),$ $\eta,\xi\in\clv_0.$ We have

\begin{itemize}
\item $\lim_{n\rightarrow\infty}-\frac{1}{2}G_nR^*R G_n\omega=-\frac{1}{2}R^*R\omega,$
\item $\lim_{n\rightarrow\infty}R G_n\omega=R\omega~;$
\end{itemize}
for $\omega\in Dom(\cll_2).$ Existence of the limit in (\ref{first limit}) now follows from the above two limits. Thus by Theorem 7.2.1  in page 174 of \cite{dgkbs}, there exists a contractive cocycle $(U_t)_{t\geq0}$ satisfying the QSDE in (\ref{QSDE}). We will prove that the coefficients associated to the QSDE in (\ref{QSDE}), satisfy the hypotheses of Theorem 7.2.3 in page 179 of \cite{dgkbs}. Hence it will follow that $(U_t)_{t\geq0}$ is an unitary cocycle, which will give the required HP dilation of the semigroup $(T_t)_{t\geq0}.$

Since the coefficient matrix is of the form
$Z=\begin{pmatrix}
-\frac{1}{2}R^*R&R^*\\
R&0
\end{pmatrix},
$ hypotheses (i) and (ii) of Theorem 7.2.3 in \cite[p.179]{dgkbs} will hold for $\clz,$ once we prove that the minimal QDS associated with the map
\begin{equation*}
\cll(x)=R^*(x\ot1_{k_0})R-\frac{1}{2}R^*R x-\frac{1}{2}xR^*R
\end{equation*}
for $x\in Dom(\cll_2),$ is conservative (see condition (v) of Theorem 3.2.16 in p.47 of \cite{dgkbs}).

Let $(\widetilde{T}_t)_{t\geq0}$ denote the minimal semigroup associated with the above map and suppose $\tilde{\cll}$ be its generator. We claim that $Dom(\cll_2)\subseteq Dom(\tilde{\cll}).$ Fix any $a\in Dom(\cll_2).$ Let $\cld$ denote the linear span of operators of the form $(1+R^*R)^{-1}\sigma(1+R^*R)^{-1}$ for $\sigma$ belonging to $B_1(L^2(tr_{_{\clh}})).$ Let $tr$ denote the canonical trace of $B(L^2(tr_{_{\clh}}))$. Using explicit forms of $\cll$ and $\tilde{\cll},$ we see that $tr(\cll(a)\rho)=tr(a\widetilde{\cll}_\ast(\rho))$ for $\rho\in\cld,$ where $\widetilde{\cll}_\ast$ denote the generator of the predual semigroup of $(\widetilde{T}_t)_{t\geq0}.$~It is known (by Lemma 3.2.5 in p.42 of \cite{dgkbs}) that $\cld$ is a core for $\widetilde{\cll}.$ So we have $tr(\cll(a)\rho)=tr(a\widetilde{\cll}_\ast(\rho))$ for all $\rho\in Dom(\widetilde{\cll}_\ast).$ Following the proof of Lemma 8.1.22 in p.204 of \cite{dgkbs}, we have $\widetilde{\cll}(a)=\cll(a).$ This implies that $Dom(\cll_2)\subseteq Dom(\widetilde{\cll})$ and as $Dom(\cll_2)$ is a core for $\cll,$ we have $Dom(\cll)\subseteq Dom(\widetilde{\cll})$ and 
$\widetilde{\cll}(a)=\cll(a)$ for all $a\in Dom(\cll).$ Now the symmetric QDS $(T_t)_{t\geq0}$ is conservative. Thus we have $1\in Dom(\cll)$ and $\cll(1)=0$ which implies that $\widetilde{\cll}(1)=0.$ Thus the minimal semigroup 
$(\widetilde{T}_t)_{t\geq0}$ is conservative. Hence by Theorem 7.2.3 in p.179 of \cite{dgkbs}, the cocycle $(U_t)_{t\geq0}$ is unitary, which completes the proof.

\end{proof}

\section*{Acknowledgement}
I thank my supervisor Debashish Goswami for   continuous encouragement and the useful discussions that I had with him.

\end{document}